\documentclass[12pt]{amsproc}

\newtheorem{theorem}{Theorem}[section]
\newtheorem{lemma}[theorem]{Lemma}
\newtheorem{corollary}[theorem]{Corollary}
\newtheorem{proposition}[theorem]{Proposition}

\theoremstyle{definition}
\newtheorem{definition}[theorem]{Definition}

\theoremstyle{remark}

\usepackage{graphicx}
\numberwithin{equation}{section}

\begin{document}

\title{ Rigidity of Polyhedral Surfaces, III
}

\author{Feng Luo}

\address{Department of Mathematics, Rutgers University, New Brunswick, New Jersey 08854}
\email{fluo@math.rutgers.edu }
\thanks{The work is supported in part by a NSF Grant.}


\date{Oct. 1, 2010.}

\dedicatory{To Dennis Sullivan on the occasion of his seventieth
birthday}


\begin{abstract} This paper investigates several global rigidity
issues for polyhedral surfaces including inversive distance circle
packings. Inversive distance circle packings are polyhedral
surfaces introduced by P. Bowers and K. Stephenson in \cite{BS} as
a generalization of Andreev-Thurston's circle packing. They
conjectured that inversive distance circle packings are rigid.
Using a recent work of R. Guo \cite{Guo} on variational principle
associated to the inversive distance circle packing, we prove
rigidity conjecture of Bowers-Stephenson in this paper. We also
show that each polyhedral metric on a triangulated surface is
determined by various discrete curvatures introduced in
\cite{Lu1}, verifying a conjecture in \cite{Lu1}. As a
consequence, we show that the discrete Laplacian operator
determines a Euclidean polyhedral metric up to scaling.
\end{abstract}

\maketitle

\section{Introduction}
\subsection{    }

This is a continuation of the study of polyhedral surfaces
\cite{Lu1}, \cite{Lu2}. The paper focuses on inversive distance
circle packings introduced by Bowers and Stephenson and several
other rigidity issues. Using a recent work of Ren Guo \cite{Guo},
we prove a conjecture of Bowers-Stephenson that inversive distance
circle packings are rigid. Namely, a Euclidean inversive distance
circle packing on a compact surface is determined up to scaling by
its discrete curvature. This generalizes an earlier result of
Andreev \cite{An} and Thurston \cite{Th} on the rigidity of circle
packing with acute intersection angles. In \cite{Lu1}, using
2-dimensional Schlaefli formulas, we introduced two families of
discrete curvatures for polyhedral surfaces and conjectured that
each of one these discrete curvatures determines the polyhedral
metric (up to scaling in the Euclidean case). We verify this
conjecture in the paper. One consequence is that for a Euclidean
or spherical polyhedral metric on a surface, the cotangent
discrete Laplacian operator determines the metric (up to scaling
in the case of Euclidean metric).  The theorems are proved using
variational principles and are based on the work of \cite{Guo} and
\cite{Lu1}. The main idea of the paper comes from reading of
\cite{BPS}, \cite{CV} and \cite{Ri}.

\subsection{ }
Recall that a \it Euclidean (or spherical or hyperbolic)
polyhedral surface \rm is a triangulated surface with a metric,
called a \it polyhedral metric, \rm so that each triangle in the
triangulation is isometric to a Euclidean (or spherical or
hyperbolic) triangle.  To be more precise, let $\bold E^2$, $\bold
S^2$ and $\bold H^2$ be the Euclidean, the spherical and the
hyperbolic 2-dimensional geometries.  Suppose $(S, T)$ is a closed
triangulated surface so that $T$ is the triangulation, $E$ and $V$
are the sets of all edges and vertices. A $K^2$  ($K^2$ = $\bold
E^2$, or $\bold S^2$, or $\bold H^2$) \it polyhedral metric \rm on
$(S, T)$ is a map $l: E \to \bold R$ so that whenever $e_i, e_j,
e_k$ are three edges of a triangle in $T$, then
$$ l(e_i) + l(e_j) > l(e_k),$$
and if $K^2=\bold S^2$, in addition to the inequalities above, one
requires
$$l(e_i)+l(e_j) + l(e_k) < 2\pi.$$
Given $l: E \to \bold R$ satisfying the inequalities above, there
is a metric on the surface $S$, called \it a polyhedral metric,
\rm so that the restriction of the metric to each triangle is
isometric to a triangle in $K^2$ geometry and the length of each
edge $e$ in the metric is $l(e)$. We also call $l: E \to \bold R$
the \it edge length function. \rm For instance, the boundary of a
generic convex polytope in the 3-dimensional space $\bold E^3$, or
$\bold S^3$ or $\bold H^3$ of constant curvature $0, 1,$ or $ -1$
is a polyhedral surface.
The \it discrete curvature $k$ \rm of a polyhedral surface is a
function $k:V \to \bold R$ so that $k(v) = 2\pi -\sum_{i=1}^m
\theta_i$ where $\theta_i$'s are the angles at the vertex $v$. See
figure 1.

\medskip
Since the discrete curvature is built from inner angles of
triangles, we consider inner angles of triangles as the basic unit
of measurement of curvature. Using inner angles, we introduce
three families of curvature like quantities in \cite{Lu1}.  The
relationships between the polyhedral metrics and curvatures are
the  focus of the study in this paper.

\medskip
\begin{definition} (\cite{Lu1}) Let $h \in \bold R$. Given a $K^2$ polyhedral metric  on $(S, T)$
where $K^2$ $=\bold E^2$, or $\bold S^2$ or $\bold H^2$, the \it
$\phi_{ h  }$ curvature \rm of a polyhedral metric is the function
$\phi_{ h  }: E \to \bold R$ sending an edge $e$ to:

\begin{equation}
\phi_{ h  }(e) = \int_{\pi/2}^{a} \sin^{ h  }(t) dt +
\int_{\pi/2}^{a'} \sin^{ h  }(t) dt
 \end{equation} where $a, a'$
are the inner angles facing the edge $e$. See figure 1.

The \it $\psi_{ h  }$ curvature \rm of the metric $l$ is the
function $\psi_{ h  }: E \to \bold R$ sending an edge $e$ to
\begin{equation} \psi_{ h  }(e) =\int^{\frac{b+c-a}{2}}_0 \cos^{ h  }(t) dt
+ \int^{\frac{b'+c'-a'}{2}}_0 \cos^{ h  }(t) dt
 \end{equation}
where $b,b',c,c'$ are inner angles adjacent to the edge $e$ and
$a,a'$ are the angles facing the edge $e$. See figure 1.
\end{definition}
\medskip
\begin{figure}[ht!]
\centering
\includegraphics[scale=0.7]{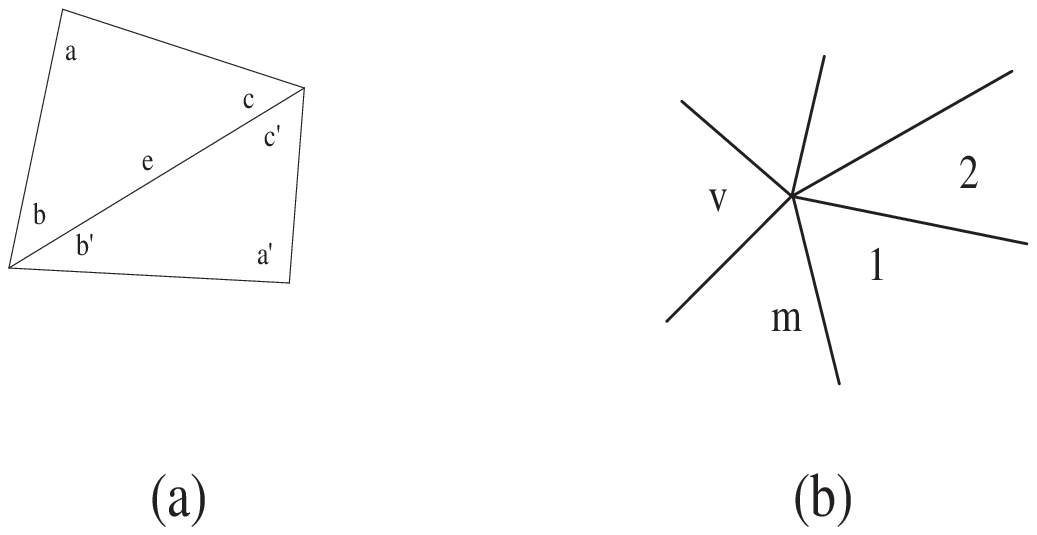}
\caption{} \label{figure 1.1}
\end{figure}


 The curvatures $\phi_0$ and $\psi_0$
were first introduced by I. Rivin [Ri] and G. Leibon [Le]
respectively. If the surface $S = \bold S^2$, then these
curvatures are essentially the dihedral angles of the associated
3-dimensional hyperbolic polyhedra at edges. The curvature
$\phi_{-2}(e) =- \cot(a) - \cot(a')$ is the discrete (cotangent)
Laplacian operator on a polyhedral surface derived from the finite
element approximation of the smooth Beltrami Laplacian on
Riemannian manifolds.

One of the remarkable theorems proved by Rivin \cite{Ri} is that a
Euclidean polyhedral metric on a triangulated surface is
determined up to scaling by its $\phi_0$ discrete curvature. In
particular, he proved that
 an ideal convex hyperbolic polyhedron is
 determined up to isometry by its dihedral angles.

We prove,

\begin{theorem}
Let $(S, T)$ be  a closed triangulated connected surface. Then for
any $h \in \bold R$,

(1) a Euclidean polyhedral metric on $(S, T)$ is determined up to
 isometry and scaling  by its  $\phi_{ h  }$
curvature.

(2) a spherical polyhedral metric on $(S, T)$ is determined up to
 isometry by its $\phi_{ h  }$ curvature.

(3) a hyperbolic polyhedral surface is determined up to isometry
by its $\psi_{ h  }$ curvature.
\end{theorem}

\medskip
We remark that theorem 1.2(1) for $h=0$ was aforementioned Rivin's
theorem. However, our proof of Rivin's theorem is different from
that in \cite{Ri} and we use the variational principle established
by Cohen-Kenyon-Propp \cite{CKP}. Theorem 1.2(3) for $h=0$ was
first proved by Leibon \cite{Le}. Theorem 1.2(2) for $h=0$ was
proved in \cite{Lu3} and theorem 1.2(2) and (3) for $h \leq -1$ or
$h \geq 0$ was proved in \cite{Lu1}.

Take $h=-2$ in theorem 1.2, we obtain,

\begin{corollary}
(1) A connected Euclidean polyhedral surface is determined up to
scaling by its discrete Laplacian operator.

(2) A spherical polyhedral surface is determined by its discrete
Laplacian operator.
\end{corollary}

Note that for a Euclidean polyhedral surface, $\phi_h = \psi_{h}$.
There remain two questions on whether  $\phi_h$ curvature
determines a hyperbolic polyhedral surface or whether  $\psi_h$
curvature determines a spherical polyhedral surface. It seems the
results may still be true in these cases.

\subsection{}
Inversive distance circle packings are polyhedral metrics on a
triangulated surface introduced by Bowers and Stephenson in
\cite{BS}. An expansion of the discussion of \cite{BS} is in
\cite{BH}. See also \cite{St}. They are generalizations of Andreev
and Thurston's circle packings. Unlike the case of Andreev and
Thurston where adjacent circles are intersecting, Bowers and
Stephenson allow adjacent circles to be disjoint and measure their
relative positions by the inversive distance. As observed in
\cite{BS}, this relaxation of intersection condition is very
useful for practical applications of circle packing to many
fields, including medical imaging and computer graphics. Based on
extensive numerical evidences, they conjectured the rigidity and
convergence of inversive distance circle packings in \cite{BS}.
Our result shows that Bowers-Stephenson's rigidity conjecture
holds.  The proof is based on a recent work of Ren Guo \cite{Guo}
which established a variational principle for inversive distance
circle packings.  A very nice geometric interpretation of the
variational principle was given in \cite{Gl}.

We begin with a brief recall of the inversive distance in
Euclidean, hyperbolic and spherical geometries. See \cite{BH} for
a more detailed discussion.  Let $K^2$ be $\bold E^2$, or $\bold
H^2$ or $\bold S^2$. Given two circles $C_1, C_2$ in $K^2$
centered at $v_1, v_2$ of radii $r_1$ and $r_2$ so that $v_1, v_2$
are of distance $l$ apart, the inversive distance $I = I(C_1,
C_2)$ between the circles is given by

\begin{equation}
 I = \frac{l^2 -r_1^2 -r_2^2}{2 r_1 r_2}
\end{equation}
in the  Euclidean plane,

\begin{equation}
 I = \frac{\cosh(l) - \cosh(r_1) \cosh(r_2)}{\sinh(r_1)
 \sinh(r_2)}
\end{equation}
in the hyperbolic plane and

\begin{equation}
I =  \frac{\cos(l) - \cos(r_1) \cos(r_2)}{\sin(r_1) \sin(r_2)}
\end{equation}
in the 2-sphere. See \cite{Guo} for more details on (1.4) and
(1.5).  If one considers $\bold E^2$, $\bold H^2$ and $\bold S^2$
as appeared in the infinity of the hyperbolic 3-space $\bold H^3$,
then $C_1$ and $C_2$ are the boundary of two totally geodesic
hyperplanes $D_1$ and $D_2$. The inversive distance $I$ is
essentially the hyperbolic distance (or the intersection angle)
between $D_1$ and $D_2$.  In particular, for the Euclidean plane
$\bold E^2$, the inversive distance $I(C_1, C_2)$ is invariant
under the inversion and hence the name.



Bowers and Stephenson's construction of an \it inversive distance
circle packing \rm with prescribed inversive distance on a
triangulated surface $(S, T)$ is as follows.  Fix once and for all
a vector $I \in \bold [-1, \infty)^E$, called the inversive
distance.

In the Euclidean case, for any $r \in \bold R_{>0}^V$, called \it
the radius vector, \rm define the edge length function $l \in
\bold R_{>0}^E$ by the formula

\begin{equation}
 l(e) =\sqrt{ r(v)^2 + r(u)^2 + 2r(v) r(u) I(e)}
\end{equation}
where the end points of the edge $e$ is $\{u,v\}$. If $l(e)$'s
satisfy the triangular inequalities that
\begin{equation}
l(e_i) + l(e_j) > l(e_k)
\end{equation}
for three edges $e_i, e_j, e_k$ of each triangle in $T$, then the
length function $l: E \to \bold R$ sending $e$ to $l(e)$ defines a
Euclidean polyhedral metric on $(S, T)$ called the \it inversive
distance circle packing  \rm with inversive distance $I(e)$ at
edge $e$. Note that if $I(e) \in [0,1]$ for all $e$, the
polyhedral metric is the circle packing investigated by Andreev
and Thurston where the intersection angle between two circles at
the end points of an edge is $\arccos(I(e))$.

In the hyperbolic geometry, one uses
\begin{equation}
 l(e) = \cosh^{-1}(\cosh(r(v)) \cosh (r(u)) + I(e) \sinh(r(v))
 \sinh(r(u))
\end{equation}
as the length of an edge. If (1.7) holds, then the lengths
$l(e)$'s define a \it hyperbolic inversive distance circle packing
\rm with inversive distances $I$ on $(S, T)$.  The \it spherical
inversive distance circle packing \rm is defined similarly with
additional condition on $l(e)$'s that
$$ l(e_i) + l(e_j) + l(e_k) < 2\pi$$ for each triangle with edges
$e_i, e_j, e_k$.

The geometric meaning of these polyhedral metrics is the
following. In each metric, if one draws a circle of radius $r(v)$
at each vertex $v$, then inversive distance of two circles at the
end points of an edge $e$ is the given number $I(e)$.

Our result which solves Bowers-Stephenson's rigidity conjecture is
the following.

\begin{theorem}
Given a closed triangulated connected surface $(S, T)$ with the
set of edges $E$ and $I \in \bold R_{\geq 0}^E$ considered as the
inversive distance,

(1) a hyperbolic  inversive distance circle packing metric on
$(S,T)$ of inversive distance $I$ is determined by its discrete
curvature $k: V \to \bold R$.

(2) an Euclidean inversive distance circle packing metric on
$(S,T)$ of inversive distance $I$ is determined by its discrete
curvature $k: V \to \bold R$ up to scaling.

\end{theorem}

Note that for $I \in [0, 1]^E$, the above result was
Andreev-Thurston's rigidity for circle packing with intersection
angles between $[0, \pi/2]$.  It seems the similar result may be
true for $I \in [-1, \infty)^E$.

\subsection{} The paper is organized as follows. In \S2, we prove
an extension lemma for angles of triangles.  We also establish a
criterion for extending a locally convex function to convex
function.  In \S3, we prove theorem 1.4. Theorem 1.2 is proved in
\S4.

The following notations and conventions will be used in the paper.
We use $\bold R$, $\bold R_{>0}$, $\bold R_{ \geq 0}$, $\bold
R_{<0}$ to denote the sets of all real numbers, positive real
numbers, non-negative real numbers, and negative real numbers
respectively.  If $X$ is a set, $\bold R^X =\{ f: X \to \bold R$\}
is the vector space of all functions on $X$. If $A$ is a subspace
of a topological space $X$, then the closure of $A$ in $X$ is
denoted by $\bar{A}$.

We thank Ren Guo for comments and careful reading of the
manuscript.

\section{Convex Extension of Locally Convex Functions}

\subsection{Continuous extension by constants}

\begin{definition} Suppose $A$ is a subspace of a topological
space $X$ and $f:A \to Y$ is continuous. If there exists a
continuous function $F: X \to Y$ so that $F|_A =f$ and $F$ is a
constant function on each connected component of $X-A$, then we
say $f$ can be \it extended continuously by  constant functions
\rm to $X$.
\end{definition}

Note that if each connected component of $X-A$ intersects the
closure of $A$, then the extension function $F$ is uniquely
determined by $f$.

The key observation of the paper is the following simple lemma.

\begin{lemma}
Suppose $\Delta$ is a triangle in the Euclidean plane $\bold E^2$,
or the hyperbolic plane $\bold H^2$, or the 2-sphere $\bold S^2$
so that its edge lengths are $l_1, l_2, l_3$ and its inner angles
are $\theta_1, \theta_2, \theta_3$. Assume that $\theta_i$'s angle
is opposite to the edge of length $l_i$ for each $i$. Consider
$\theta_i =\theta_i(l)$ as a function of $l=(l_1, l_2, l_3)$.

\begin{enumerate}
\item If $\Delta$ is Euclidean or hyperbolic, the angle function
$\theta_i$ defined on $\Omega =\{ (l_1, l_2, l_3) \in \bold R^3 |
l_1+l_2
> l_3, l_2+l_3 > l_1, l_3+l_1>l_2\}$ can be extended
continuously by constant functions to a function
$\tilde{\theta_i}$ on $\bold R^3_{>0}$.

\item If $\Delta$ is spherical, the angle function $\theta_i$
defined on $\Omega =\{ (l_1, l_2, l_3) \in \bold R^3 | l_1+l_2
> l_3, l_2+l_3 > l_1, l_3+l_1>l_2, l_1+l_2+l_3 < 2\pi\}$ can be extended
continuously by constant functions to a function
$\tilde{\theta_i}$ on $\bold (0, \pi)^3$.

\end{enumerate}
\end{lemma}

We call the set $\Omega$ in the lemma the \it natural domain \rm
of the length vectors.


\begin{proof}
In the case (1), the extension function $\tilde{\theta_i}$ of
$\theta_i$ is given by  $\tilde{\theta_i} = \pi$ when $l_i \geq
l_j+l_k$, and $\tilde{\theta_i} =0$ when $l_j \geq l_i+l_k$. It
remains to verify the continuity of $\tilde{\theta_i}$ on $\bold
R_{>0}^3$. It is based on the cosine law. Given a point $L =(L_1,
L_2, L_3)$ in the boundary $\bar{\Omega} -\Omega$ of $\Omega$
inside $\bold R^3_{>0}$, we may assume without loss of generality
that $L_1=L_2+L_3$. The continuity of $\tilde{\theta_i}$ follows
from

$$\lim_{ l \to L} \theta_1(l) =\pi,  \quad \lim_{l \to L}
\theta_j(l) =0, \quad j=2,3.$$

Indeed, the cosine law says, in the case of $\Delta \subset \bold
E^2$, that \begin{equation} \cos(\theta_i) =\frac{ l_j^2 + l_k^2
-l_i^2}{ 2 l_j l_k}. \end{equation}

One sees easily that when $l$ tends to $L$, then the
right-hand-side of (2.1) tends to $1$ if i=2,3 and $-1$ if $i=1$.
This verifies the continuity in the Euclidean case. In the
hyperbolic case, the cosine law says

\begin{equation}
\cos(\theta_i) =\frac{ \cosh(l_j)\cosh(l_k)
-\cosh(l_i)}{\sinh(l_j)\sinh(l_k)}.\end{equation} Thus one sees
that as $l$ tends to $L=(L_1, L_2, L_3)$ with $L_j >0$, the right-hand-side of (2.2) tends to
$1$ if $i=2,3$ and to $-1$ if $i=1$. Thus $\tilde{\theta_i}$ is
continuous.

To see (2), recall that the cosine law for spherical triangle says
\begin{equation}
\cos(\theta_i) =\frac{\cos(l_i) - \cos(l_j)\cos(l_k)
}{\sin(l_j)\sin(l_k)}.
\end{equation}
If $l$ tends to $L$ where $L_1=L_2+L_3$ with $L_i \in (0,\pi)$,
then $\lim_{l \to L} \cos(\theta_1) =-1$ and $\lim_{l \to L}
\cos(\theta_i) =1$ when $i=2,3$.  On the other hand, if
$L_1+L_2+L_3 =2\pi$ for $L_i \in (0, \pi)$, then the cosine law
implies that $\lim_{l \to L} \cos(\theta_i) =-1$ for all $i$,
i.e., all inner angles are $\pi$ in this case. Thus  by setting
the extended function $\tilde{\theta_i}$ in $(0, \pi)^3$ to be
$\tilde{\theta_i}(l) = \pi$ if $l_i \geq l_j + l_k$,
$\tilde{\theta_i}(l) = 0$ if $l_j \geq l_i + l_k$, and
$\tilde{\theta_i}(l) = \pi$ if $l_i + l_j + l_k \geq \pi$, (
$\{i,j,k\}=\{1,2,3\}$), we see that  $\tilde{\theta_i}$ is
continuous.

\end{proof}

\subsection{Continuous extension of 1-forms and
 of locally convex functions}

We establish some simple facts on extending closed 1-forms and
 locally convex functions to convex functions in this
subsection.

\begin{definition} A differential 1-form $w =\sum_{i=1}^n a_i(x) dx_i$
in an open set $U \subset \bold R^n$ is said to be \it continuous
\rm if each $a_i(x)$ is a continuous function on $U$. A continuous
1-form $w$ is called \it closed \rm if $\int_{\partial \tau} w=0$
for each Euclidean triangle $\tau$ in $U$.
\end{definition}

By the standard approximation theory, if $w$ is closed and
$\gamma$ is a piecewise smooth null homologous loop in $U$, then
$\int_{\gamma} w =0$. If $U$ is simply connected, then the
integral $F(x) = \int_a^x w$ is well defined, independent of the
choice of piecewise smooth paths in $U$ from $a$ to $x$. The
function $F(x)$ is $C^1$-smooth so that $\partial F(x) /
\partial x_i = a_i(x)$.


\begin{proposition} Suppose $X$ is an open set in $\bold
R^n$ and $A \subset X$ is an open subset bounded by a smooth
(n-1)-dimensional submanifold in $X$. If $w=\sum_{i=1}^n a_i(x)
dx_i$ is a continuous 1-form on $X$ so that $w|_A$ and $w|_{X
-\bar{A}}$ are closed where $\overline{A}$ is the closure of $A$
in $X$, then $w$ is closed in X.
\end{proposition}

\begin{proof}
Since closedness is a local property and is invariant under smooth
change of coordinates in $X$, we may assume that $X=\bold R^n$ and
$A =\{ (x_1, ..., x_n) \in \bold R^n | x_n >0\}$. Take a Euclidean
triangle $\tau \subset X$. To verify $\int_{\partial \tau } w=0$,
we may assume that $\tau$ is not in $\overline{A}$ or  $X-A$ since
otherwise $\int_{\partial \tau} w=0$ follows from the assumption
and the standard approximation theorem. In the remaining case,
$\tau$ intersects both $A$ and $X-A$.  The plane $x_n=0$ cuts the
triangle $\tau$ into a triangle $\gamma_1$ and a quadrilateral
$\gamma_2$ so that $\gamma_1$ and $\gamma_2$ are in the closure of
$A$ and $X-A$. We can express, in the singular chain level,
$\partial \tau = \partial \gamma_1 +
\partial \gamma_2$.
By definition, $\int_{\partial \gamma_i} w=0$ for each $i$. Thus
$\int_{\partial \tau} w = \int_{\partial \gamma_1} w +
\int_{\partial \gamma_2} w =0$.
\end{proof}

A \it real analytic codimension-1 submanifold \rm $Y$ in an open
set $X$ in $\bold R^n$ is a smooth submanifold so that locally $Y$
is defined by $k(x)=0$ for a non-constant real analytic function
$k$. Note that if $L$ is a (compact) line segment in $X$, then
either $L \subset Y$ or $L \cap Y$ is a finite set. This is due to
the fact that a non-constant real analytic function on an open
interval has isolated zeros.

Recall that a function $f$ defined on a convex set $X \subset
\bold R^n$ is called \it convex \rm if for all $p, q \in X$ and
all $t \in [0, 1]$, $tf(p) + (1-t) f(q) \geq f(tp+(1-t)q)$. It is
called \it strictly convex \rm if for all $p \neq q$ in $X$ and
all $t \in (0,1)$, $tf(p) + (1-t) f(q) > f(tp+(1-t)q)$.  A
function $f$ defined in an open set $U \subset \bold R^n$ is said
to be \it locally convex \rm (or \it locally strictly convex \rm)
if it is convex (or strictly convex) in a convex neighborhood of
each point.

\begin{proposition} Suppose $X \subset \bold R^n$ is an open
convex set and $A \subset X$ is an open subset of $X$ bounded by a
codimension-1 real analytic  submanifold  in $X$. If
$w=\sum_{i=1}^n a_i(x) dx_i$ is a continuous closed 1-form on $X$
so that $F(x) =\int^x_a w$ is locally convex in $A$ and in  $X
-\overline{A}$, then $F(x)$ is convex  in $X$.
\end{proposition}

\begin{proof} Since $X$ is simply connected, the function $F$ is
well defined. To verify convexity, take $p, q \in X$ and consider
$f(t) = F(tp+(1-t)q)$ for $t \in [0,1]$. It suffices to show that
$f(t)$ is convex in $t$.  Since $F$ is $C^1$-smooth, $f$ is
$C^1$-smooth. Let $\partial A =\bar{A} -A$ and $L$ be the line
segment from $p$ to $q$. Since $\partial A$ is real analytic,
either $L$ intersects $\partial A$ in a finite set of points, or
$L$ is in $\partial A$. In the first case, let $0=t_0< t_1 < ...,
t_n =1$ be the partition of $[0,1]$ so that the line segment $tp +
(1-t)q$ for $t \in (t_i, t_{i+1})$ is either in $A$ or in
$X-\overline{A}$.  By definition, $f(t)$ is convex in $[t_i,
t_{i+1}]$, i.e., $f'(t)$ is increasing in $[t_i, t_{i+1}]$ for
$i=0,..., n-1$. Since $f'(t)$ is continuous in $[0,1]$, this
implies that $f'(t)$ is increasing in $[0,1]$, i.e., $f(t)$ is
convex in $[0,1]$. In the second case that $L \subset
\partial A$, we take two  sequences of points $p_m$ and
$q_m$ converging to $p$ and $q$ respectively in $X$ so that $p_m,
q_m$ are not in $\partial A$. Then by the case just proved, the
functions $f_m (t) = F( tp_m + (1-t) q_m)$ are convex in $t$.
Furthermore, $f_m$ converges to $f$. Thus $f$ is convex.
\end{proof}

\begin{corollary} Suppose $X \subset \bold R^n$ is an open
convex set and $A \subset X$ is an open subset of $X$ bounded by a
real analytic codimension-1 submanifold in $X$. If $w=\sum_{i=1}^n
a_i(x) dx_i$ is a continuous closed 1-form on $A$ so that $F(x)
=\int^x_a w$ is locally convex on $A$  and each $a_i$ can be
extended continuously to $X$ by constant functions to a function
$\tilde{a_i}$ on $X$, then $\tilde{F}(x) = \int^x_a \sum_{i=1}^n
\tilde{a_i} dx_i$ is a $C^1$-smooth convex function on $X$
extending $F$.
 \end{corollary}

We remark that the real analytic assumption in the proposition 2.5
can be relaxed to $C^2$ smooth.

\section{A Proof of Bowers-Stephenson's Rigidity Conjecture}

We begin by recalling Guo's work on a variational principle
associated to inversive distance circle packings and then prove
theorem 1.4. We will work in Euclidean and hyperbolic geometries
only.

\subsection{Guo's variational principle for inversive distance
circle packing}

Suppose $\Delta$ is a triangle with vertices $v_1, v_2, v_3$ and
edges $e_{ij} = v_i v_j$, $ i \neq j$.  Fix once and for all an
inversive distance $I_{ij} \in [0, \infty)$ at each edge $e_{ij}$.
Then for each assignment of positive number $r_i$ at $v_i$ for
$i=1,2,3$,  let
\begin{equation} l_{k} =\sqrt{ r_i^2 + r_j^2 + 2 r_i r_j I_{ij}}
\end{equation} for
 Euclidean geometry and
 \begin{equation} l_{k} = \cosh^{-1} ( \cosh(r_i)
\cosh(r_j) + I_{ij} \sinh(r_i) \sinh(r_j)) \end{equation} for
hyperbolic geometry where $\{i,j,k\}=\{1,2,3\}$.

Let $\Omega =\{ (x_1, x_2, x_3)\in \bold R_{>0}^3 | x_i + x_j >
x_k , \{i,j,k\}=\{1,2,3\} \}$. If $(l_1, l_2, l_3)$ is in
$\Omega$, then we  construct a Euclidean triangle $\Delta$ with
length $l_k$ of $e_{ij}$  given by (3.1) and a hyperbolic
triangle, still denoted by $\Delta$, with  length $l_{k}$ of
$e_{ij}$ given by (3.2). Suppose the angle of the triangle at
$v_i$ is $\theta_i$ and consider
 $\theta_i$ as a
function of $(r_1, r_2, r_3)$.  Guo proved the following theorem
in \cite{Guo}.

\begin{theorem} (Guo \cite{Guo})  Fix any $(I_{12}, I_{23},
I_{31}) \in [0, \infty)^3$.

(1) For Euclidean triangles, let $u_i=\ln r_i$, then the
differential 1-form $w=\sum_{i=1}^3 \theta_i d u_i$ is closed in
the open subset of $\bold R^3$ where it is defined. The integral
$F(u)=\int_0^u w$ is a locally concave function in $u=(u_1, u_2,
u_3)$ and is strictly locally concave in $u_1+u_2+u_3=0$.
Furthermore, if $c \in \bold R$ and $F(u)$ is defined, then
$F(u+(c,c,c))=F(u)$.

(2)  For hyperbolic triangles, let $u_i=\ln (\tanh(r_i/2))$, then
the differential 1-form $w=\sum_{i=1}^3 \theta_i d u_i$ is closed
in the open subset of $\bold R_{<0}^3$ where it is defined.
Furthermore, the integral $F(u)=\int_{-(1,1,1)}^u w$ is a strictly
locally concave function in $u=(u_1, u_2, u_3)$.
\end{theorem}

It is also proved in \cite{Guo} that the open sets where the
1-forms $w$ are defined in theorem 3.1 are connected and simply
connected.  Theorem 3.1 is a generalization of an earlier result
obtained in \cite{CL}. Guo proved a local and infinitesimal
rigidity theorem for inversive distance circle packing using
theorem 3.1. It says that a Euclidean inversive distance circle
packing is locally determined, up to scaling, by the discrete
curvature of the underlying polyhedral surface. He also proved the
local and infinitesimal rigidity for hyperbolic inversive distance
circle packings.

\subsection{Concave extension of Guo's action
functional}
 Our main observation is that Guo's
differential 1-forms $w = \sum_{i=1}^3 \theta_i du_i$ can be
extended to a closed 1-form on $\bold R^3$ in the Euclidean case
and on $\bold R^3_{<0}$ in  the hyperbolic case so that the
integrations of the extended 1-forms are still concave.

\begin{proposition}
Let $w$ be the 1-forms defined in theorem 3.1.

(a) In the case of Euclidean triangles, the 1-form $w$ can be
extended to a continuous closed  1-form $\tilde{w}$ on $\bold R^3$
so that the integration $\tilde{F}(u) = \int^u_0 \tilde{w}$ is a
$C^1$-smooth concave function.

(b) In the case of hyperbolic triangles, the 1-form $w$ can be
extended to a continuous closed  1-form $\tilde{w}$ on $\bold
R_{<0}^3$ so that the integration $\tilde{F}(u) =
\int^u_{-(1,1,1)} \tilde{w}$ is a $C^1$-smooth concave function.
\end{proposition}

We begin by focusing the 1-forms in its radius coordinate $r=(r_1,
r_2, r_3) \in \bold R_{>0}^3$. In this case, the 1-forms are given
by $w=\sum_{i=1}^3 \theta_i \frac{d r_i}{r_i}$ and $w=\sum_{i=1}^3
\theta_i \frac{ dr_i}{\sinh(r_i)}$. The 1-form $w$ is defined on
the open set $U$ of $\bold R_{>0}^3$ where
\begin{equation}
 U= \{ (r_1, r_2, r_3) \in \bold R_{>0}^3 |
l_i+l_j > l_k, \{ i,j,k \}=\{1,2,3\} \},
\end{equation}
where $l_i = l_i(r_1, r_2,r_3)$ is defined on $\bold R_{>0}^3$.
(Note that for hyperbolic and Euclidean geometries, the sets $U$
are different due to (3.1) and (3.2)). The extension of the 1-form
$w$ is the natural one. Namely, we replace $\theta_i$ in $w$ by
$\tilde{\theta_i}$ appeared in lemma 2.1. Thus the extended 1-form
is $\tilde{w} =\sum_{i=1}^3 \tilde{\theta_i} \frac{d r_i}{r_i}$ or
$\tilde{w} =\sum_{i=1}^3 \tilde{\theta_i} \frac{d
r_i}{\sinh(r_i)}.$

It remains to show that $\tilde{w}$ is continuous and closed in
$\bold R_{>0}^3$ so that its pull back to the $u$-coordinate has a
concave integration. To this end, we prove,

\begin{lemma}   Let $\bar{U}$ be the closure of $U$ in
$\bold R_{>0}^3$. Then,

(1) $\theta_i$ is a constant function on each connected component
of $\bar{ U} - U$, and

(2) for each connected component $V$ of $\bold R_{>0}^3-U$, the
intersection $V \cap \bar{U}$ is a connected component of $\bar{U}
-U$.

\end{lemma}
\begin{proof}
By (3.3), the boundary $\partial U =\bar{ U} - U$ is given by
$\cup_{i=1}^3 \partial_i U$ where $\partial_i U =\{(r_1, r_2, r_3)
\in \bold R_{>0}^3 | l_i=l_j+l_k, \{j,k\}=\{1,2,3\}-\{i\} \}$.
Furthermore,  $\bold R_{>0}^3 -U =\cup_{i=1}^3 V_i$ where $V_i =
\{(r_1, r_2, r_3) \in \bold R_{>0}^3 | l_i \geq l_j+l_k,
\{j,k\}=\{1,2,3\}-\{i\} \}.$

First, we note that if $I_{ij} \leq 1$, then $\partial_k U =
\emptyset $ and $V_k =\emptyset$. Indeed, if $I_{ij} \leq 1$, then
by (3.1) and (3.2),
$$l_k \leq r_i + r_j.$$  But due to $I_{ab} \geq 0$, (3.1) and (3.2), $r_j < l_i$ and
$r_i < l_j$. Therefore, $l_k < l_i +l_j$. This implies that
$\partial_k U = \emptyset$ and $V_k = \emptyset$.

Next $\partial_i U \cap
\partial_j U = \emptyset$  and $V_i \cap V_j =\emptyset$
for $i \neq j$. Indeed, if $r \in \partial_i
U \cap
\partial_j U$ or $r \in V_i \cap V_j$, then $l_i \geq
l_j+l_k$ and $l_j \geq l_i+l_k$. Thus $l_k=0$. But $l_k > r_i
>0$.

We claim that if $I_{ij} > 1$, then both $V_k$ and $\partial_k U$
are non-empty and connected.  Assume the claim, then the lemma
follows. Indeed, since $l_s>0$ for all indices $s$, it follows, by
lemma 2.1, that $\theta_i$ is either $0$ or $\pi$ in $\partial_s
U$, i.e., (1) holds. Next, $V_s$'s are the connected components of
$\bold R_{>0}^3-U$ so that $V_s \cap \bar{U} = \partial_s U$. Thus
(2) holds.

To see the claim, it suffices to show that there is a smooth
function $f(r_i, r_j)$ defined on $\bold R_{>0}^3$ so that its
graph is $\partial_k U$ and $V_k=\{ (r_1, r_2, r_3) \in  \bold
R_{>0}^3 | 0 < r_3 \leq f(r_1, r_2)\}$.

To this end, consider the equation
\begin{equation}
l_k = l_i + l_j,
\end{equation} and let the right-hand-side of
(3.4) be $g(r_k, r_i, r_j)$.  We will deal with the Euclidean and
hyperbolic geometry separately.

CASE 1 Euclidean triangles. In this case, the function $g(r_k,r_i,
r_j)$ is given by

\begin{equation}
g(r_k, r_i, r_j) = \sqrt{ r_k^2 + r_j^2 + 2I_{kj} r_kr_j }+ \sqrt{
r_i^2 + r_k^2 + 2I_{ik} r_ir_k }
\end{equation}
Evidently, for a fixed $(r_i, r_j) \in \bold R^2_{>0}$,  $g(r_k,
r_i, r_j)$ is a strictly increasing function of $r_k \in \bold
R_{>0}$ so that $g(0, r_i, r_j) =r_i+r_j < \sqrt{ r_i^2 + r_j^2 +
2I_{ij} r_ir_j }$ (due to $I_{ij} > 1$) and $\lim_{ r_k \to
\infty} g(r_k, r_i, r_k) = \infty$. By the mean-value theorem,
there exists a unique positive number $f(r_i, r_j)$ so that
$g(f(r_i, r_j), r_i, r_j)= \sqrt{ r_i^2 + r_j^2 + 2 r_i r_j
I_{ij}} =l_k$. The smoothness of $f(r_i, r_j)$ follows from the
implicit function theorem applied to (3.4). Indeed,
$$ \frac{\partial g}{\partial r_k} = \frac{ r_k + 2 I_{kj}r_j}{l_i} + \frac{r_k+ 2 I_{ik}
r_i}{l_j} >0.$$ Thus, $f(r_i, r_j)$ is smooth.

This shows $\partial_k U$ is the graph of the smooth function $f$
defined on $\bold R_{>0}^2$, i.e., $$\partial_k U =\{(r_1, r_2,
r_3) \in \bold R_{>0}^3 | r_k = f(r_i, r_j)\}.$$ Thus it is
connected. Since $g(r_k,r_i, r_j)$ is an increasing function of
$r_k$,  $V_k =\{ r \in R_{>0}^3 | 0<r_k \leq f(r_i, r_j),
\{i,j\}=\{1,2,3\}-\{k\} \}$. Thus $V_k$ is connected.

CASE 2 hyperbolic triangles. By the same argument as in case 1, it
suffices to show the same properties established in case 1 hold
for $g(r_k, r_i, r_j)$ given by
\begin{equation}
\cosh^{-1} ( \cosh(r_i) \cosh(r_k) + I_{ik} \sinh(r_i)
\sinh(r_k))+\cosh^{-1} ( \cosh(r_k) \cosh(r_j) + I_{kj} \sinh(r_k)
\sinh(r_j)).
\end{equation}
Fix $(r_i, r_j) \in \bold R^2_{>0}$. Then the function $g(r_k,
r_i, r_j)$ is clearly strictly increasing in $r_k \in \bold
R_{>0}$ so that $\lim_{r_k \to \infty} g(r_k, r_i, r_j) =\infty$
and due to $I_{ij} >1$,

 $$g(0, r_i,r_j) =r_i + r_j$$
 $$=\cosh^{-1}( \cosh(r_i+r_j))$$
$$ = \cosh^{-1}( \cosh(r_i)\cosh(r_j) + \sinh(r_i) \sinh(r_j))$$
$$ < \cosh^{-1}( \cosh(r_i) \cosh(r_j) + I_{ij} \sinh(r_i)
\sinh(r_j)) =l_k.$$

By the mean value theorem, there exists a unique positive number
$f(r_i, r_j)$ so that $g(f(r_i, r_j), r_i, r_j)=l_k$.  The
smoothness of $f(r_i, r_j)$ follows form the implicit function
theorem that
$$\frac{\partial g}{\partial r_k} =
\frac{ \cosh(r_i)\sinh(r_k) +
I_{ik}\sinh(r_i)\cosh(r_k)}{\sqrt{(\cosh(r_i)\cosh(r_k) +
I_{ik}\sinh(r_i)\sinh(r_k))^2-1}} $$ $$+ \frac{
\cosh(r_j)\sinh(r_k)+
I_{jk}\sinh(r_j)\cosh(r_k)}{\sqrt{(\cosh(r_j)\cosh(r_k)+
I_{jk}\sinh(r_j)\sinh(r_k))^2-1}}
$$ $$>0.$$
By the same argument as in case 1, we see that $\partial_k U$,
being the graph of the smooth function $f$, is connected and
$V_k$, being the region below the positive function $f$ over
$\bold R^2_{>0}$, is also connected.
\end{proof}

Now back to the proof of proposition 3.2, for part (1), consider
the real analytic diffeomorphism $u=u(r): \bold R^3_{>0} \to \bold
R^3$ where $u_i = \ln r_i$. The differential 1-form
$w=\sum_{i=1}^3 \theta_i \frac{dr_i}{r_i}$ pulls back (via
$r=u^{-1}(r)$) to $w = \sum_{i=1}^3 \theta_i du_i$ as appeared in
theorem 3.1. By lemma 3.3, the extension $\tilde{w} = \sum_{i=1}^3
\tilde{\theta_i} du_i$ is obtained from $w$ by extending each
coefficient $\theta_i$ by constant functions on $\bold R^3 -
u^{-1}(U)$. Thus, by corollary 2.6, the function $\tilde{F}(u) =
\int^u_0 \tilde{w}$ is a $C^1$-smooth concave function in $u \in
\bold R^3$ so that
\begin{equation}
\partial \tilde{F}/\partial u_i = \tilde{\theta_i}.
\end{equation}

The same argument also works for part (2) since $u=u(r)$ with $u_i
= \ln \tanh(r_i)$ is a real analytic diffeomorphism from $\bold
R_{>0}^3$ onto $\bold R_{<0}^3$.

\subsection{A proof of theorem 1.4 for Euclidean inversive distance circle packing
} Suppose otherwise that there exist two inversive circle packing
metrics $d_1, d_2$  on $(S, T)$ with the same inversive distance
$I \in \bold [0, \infty)^E$ so that their discrete curvatures are
the same and $d_1 \neq \lambda d_2$ for any $\lambda$. Let $a \in
\bold R^V$ be their common discrete curvature.

We will use the notation that if $i \in V$ and $x \in \bold R^V$,
then $x_i =x(i)$ below. Let $T^{(2)}$ be the set of all triangles
in $T$. If a triangle $s \in T^{(2)}$ has vertices $i,j,k \in V$,
then we denote the triangle by $s=\{i,j,k\}$.  For circle packing
metrics of radii $r \in \bold R_{>0}^V$ with a given inversive
distance $I$, we use $u \in \bold R^V$ to denote their logarithm
coordinate where $u_i =\ln r_i$. Thus, there are two points $p, q$
in $\bold R^V$ as the logarithmic coordinates of $d_1$ and $d_2$
so that their discrete curvatures are $a \in \bold R^V$ and $p-q
\neq \lambda(1,1,1,..,1)$ for any $\lambda$.

We will derive a contradiction by using the locally concave
functions $F$ and its concave extension $\tilde{F} =\int^u_0
\tilde{w}$ appeared in  proposition 3.2 associated to theorem
3.1(1).

Define a $C^1$-smooth function $W: \bold R^V \to \bold R$ by
\begin{equation}
W(u) = -\sum_{s \in T^{(2)}, s=\{i,j,k\}, i,j,k \in V}
\tilde{F}(u_i, u_j, u_k) + \sum_{ i \in V} (2\pi -a_i) u_i.
\end{equation}

The function $W$ is convex since it is a summation of convex
functions. Furthermore, by the definition of $W$, (3.7), the
definition of discrete curvature $(a_i)$, $p$ and $q$ are both
critical points of $W$. Since $W$ is convex in $\bold R^V$, $p$
and $q$ are both minimal points of $W$. Furthermore, for all $t
\in [0,1]$, $tp+(1-t)q$ are minimal points of $W$. In particular,
$$ W(tp + (1-t)q) = W(p)$$ for all $ t \in [0, 1]$. Since
$$W(tp+(1-t)q) = \sum_{s \in T^{(2)}, s=\{i,j,k\}, i,j,k \in V}
f_{ijk}(t) + \sum_{ i \in E} (2\pi- a_i) (tp_i + (1-t) q_i)$$
where the function \begin{equation} f_{ijk}(t) =-\tilde{F}(tp_i +
(1-t) q_i, tp_j + (1-t)q_j , tp_k +(1-t) q_k)
\end{equation} is convex, it follows that $f_{ijk}(t)$ is linear
in $t \in [0, 1]$ for all triangle $s$ with vertices $i,j,k$. This
is due to the simple fact that a summation of a convex function
with a strictly convex function is strictly convex. By the
assumption that $p-q \neq c(1,1,....,1)$ in $\bold R^V$ and the
surface is connected,  there exists a triangle $s$ with vertices
$i,j,k \in V$ so that $(p_i, p_j, p_k) -(q_i, q_j, q_k) \neq
(c,c,c)$ for all $c \in \bold R$. By the given assumption,
$(p_i,p_j,p_k)$ and $(q_i, q_j, q_k)$ are in the domain of
definition of $w$ in theorem 3.1.  Thus for $t \in [0, 1]$ close
to $0$ or $1$, by theorem 3.1 on the local strictly convexity of
$-F(u_1, u_2, u_3)$ on $u_1+u_2+u_3=0$ and $F(u+(c,c,c))=F(u)$,
$f_{ijk}(t)$ is strictly convex in $t$ near $0, 1$. This is a
contradiction to the linearity of $f_{ijk}(t)$.

\subsection{A proof of theorem 1.4 for hyperbolic inversive distance circle packing
} The proof is essentially the same as in \S3.3 and is simpler.
For any $r \in \bold R_{>0}^V$, define $u=u(r) \in \bold R_{<0}^V$
by $u_i = \ln \tanh(r_i/2))$. For a circle packing with radii $r
\in \bold R_{>0}^V$, let $u=u(r)$ and call it the $u$-coordinate
of the circle packing metric.

We use the same notation as in \S3.3. Suppose the result does not
hold and let $p \neq q \in \bold R_{<0}^V$ be the $u$-coordinates
of the two distinct hyperbolic circle packing metrics having the
same hyperbolic inversive distance $I \in \bold R_{\geq 0}^E$ and
the same discrete curvature $a =(a_i) \in \bold R^V$. Define the
action functional $W$ on $\bold R_{<0}^V$ by the same formula
(3.8) where $\tilde{F}$ is the concave function in proposition 3.2
associated to theorem 3.1(2). Then the same proof goes through as
in \S3.3 since in this case, one of $f_{ijk}(t)$ is strictly
convex for $t$ near 0 and 1.

\section{2-dimensional Schlaefli Type Action Functionals and Their Extensions}

The following was proved in \cite{Lu1}. The proof is a straight
forward calculation.

\begin{theorem}
Suppose $\Delta$ is a triangle in the Euclidean plane $\bold E^2$,
or the hyperbolic plane $\bold H^2$, or the 2-sphere $\bold S^2$
so that its edge lengths are $l_1, l_2, l_3$ and its inner angles
are $\theta_1, \theta_2, \theta_3$ where $\theta_i$'s angle is
opposite to the edge of length $l_i$. Let $h \in \bold R$ and let
$\Omega$ be the natural domain for length vectors  appeared in
lemma 2.2.

\begin{enumerate}

\item For a Euclidean triangle,
$$ w_{ h  } = \sum_{i=1}^3 \frac{\int^{\theta_i}_{\pi/2}
\sin^{ h  }(t) dt}{l_{i}^{h +1}}  dl_i$$ is a closed 1-form on
$\Omega$. The integral  $\int^u_{-(h,h,h)} w_{ h  }$ is locally
convex in variable $u=(u_1, u_2, u_3)$ where $u_i=\ln l_i$ for
$h=0$ and $u_i=-\frac{l_i^{-h}}{h}$ for $h \neq 0$. Furthermore,
$\int^u_{-(h,h,h)} w_h$ is locally strictly convex in hypersurface
$u_1+u_2+u_3 =0$.

\item  For a spherical triangle,
$$ w_{ h  } = \sum_{i=1}^3 \frac{\int^{\theta_i}_{\pi/2}
 \sin^{ h  }(t) dt}{\sin^{ h  +1}(l_i)} dl_i$$ is a closed 1-form
 on $\Omega$.
The integral $\int^u_0 w_{ h  }$ is locally strictly convex in
$u=(u_1, u_2, u_3)$ where  $u_i =\int^{l_i}_{\pi/2}$  $ \sin^{- h
-1}(t) dt$.

\item  For a hyperbolic triangle,
$$ w_{ h  } = \sum_{i=1}^3 \frac{\int^{\theta_i}_{\pi/2}  \sin^{ h  }(t) dt}{\sinh^{ h  +1}(l_i)} dl_i$$ is a closed
 1-form.

\item  For a hyperbolic triangle,
$$ w_{ h  } = \sum_{i=1}^3 \frac{\int^{\frac{1}{2}
(\theta_i -\theta_j -\theta_k) }_0  \cos^{ h  }(t) dt}{\coth^{ h
+1}(l_i/2)} dl_i$$ is a closed  1-form. The integral $\int^u_0 w_{
h }$ is locally strictly convex in $u=(u_1, u_2, u_3)$ where  $u_i
=\int^{l_i}_{1}$  $ \coth^{- h -1}(t/2) dt$.

\end{enumerate}
\end{theorem}

\subsection{}
Recall that the natural domain $\Omega$ of the edge length vectors
is given by $\Omega =\{ (l_1, l_2, l_3) \in \bold R_{>0}^3 | l_i +
l_j > l_k, \{i,j,k\}=\{1,2,3\}\}$ for Euclidean and hyperbolic
triangles and  $\Omega = \{ (l_1, l_2, l_3) \in \bold R_{>0}^3 |
l_i + l_j
> l_k,   l_1+l_2+l_3 < 2\pi, \{i,j,k\}=\{1,2,3\} \}$.
 Let $J$ be the
natural interval for each individual length $l_i$, i.e., $J =\bold
R_{>0}$ for Euclidean and hyperbolic triangles and $J=(0, \pi)$
for spherical triangles. In each case of theorem 4.1, there exists
a real analytic diffeomorphism $g: J \to g(J)$ from $J$ onto the
open interval $g(J)$ so that $u_i = g(l_i)$. To be more precise,
$g(t) = \ln t$ in the case of $h=0$ of theorem 4.1(1), $g(t)
=-\frac{t^{-h}}{h}$ ($h \neq 0$) in the case of $h \neq 0$ in
theorem 4.1(1),  $g(t) = \int^t_{\pi/2} \sin^{-h-1}(x) dx$ in the
case (2) of theorem 4.1, $g(t)=\int^t_{1} \sinh^{-h-1}(x) dx$ in
the case (3) of theorem 4.1 and $g(t) =\int^t_1 \coth^{-h-1}(x)
dx$ in the case of (4).  The real analytic diffeomorphism $u(l_1,
l_2, l_3) =(u_1, u_2, u_3)$ where $u_i=g(l_i)$ sends $J^3$ onto
the open cube $g(J)^3$ in $\bold R^3$.

By lemma 2.2, each of the angle function $\theta_i(l): \Omega \to
\bold R$ can be extended by constant functions to a continuous
function $\tilde{\theta_i}(l): J^3 \to \bold R$.  Define a
continuous 1-form $\tilde{w_h}$ on $J^3$ by replacing $\theta_i$
in the definition of $w_h$ in theorem 4.1 by $\tilde{\theta_i}$.

\begin{lemma}
The continuous differential 1-form $\tilde{w_h}$ is closed in
$J^3$. \end{lemma}

\begin{proof} By proposition 2.4 where we take $X =J^3$ and
$A=\Omega$, it suffices to show that $\tilde{w_h}$ is closed in
each connected component $U$ of $J^3 -\overline{\Omega}$. By
theorem 4.1 $\tilde{w}|_A$ is closed, the restriction of
$\tilde{w_h}$ to $U$ is of the form $\sum_{i=1}^3 c_i du_i$ where
$u_i = g(l_i)$ and $c_i$ is a constant. Thus $\tilde{w_h}|_U$ is
closed.
\end{proof}

\begin{proposition}
The pull back 1-form $(u^{-1})^* (\tilde{w_h})$ on $g(J)^3$ is a
closed 1-form. Furthermore, if $F(u) = \int^u w_h$ is locally
convex in $u(\Omega)$ (i.e, in the case (1),(2), (4) of theorem
4.1), then $\tilde{F}(u) = \int^u (u^{-1})^* (\tilde{w_h})$ is
convex in $u$ in $g(J)^3$.
\end{proposition}

Note that by the construction, if $u \in u(\Omega)$ and $w_h =
\sum_{i=1}^3 \alpha_{i, h} (u) du_i$ (as shown in theorem 4.1)
then
\begin{equation}
\frac{\partial \tilde{F}(u)}{\partial u_i} = \alpha_{i, h}(u).
\end{equation}
Furthermore, by definition, the $\phi_h$ and $\psi_h$ curvatures
are sum of two of $a_{i,h}(u)$'s.

\begin{proof} By corollary 2.6 where we take $X = g(J)^3$ and $A=
u(\Omega)$, it suffices to show that $u(\Omega)$ is bounded by a
real analytic surface in $X$ and $\tilde{F}(u)$ is convex in
$u(\Omega)$ and in each component of $g(J)^3 -
\overline{u(\Omega)}$.

Since $\Omega$ in $J^3$ is bounded by hyperplanes and $u(l)
=(g(l_1), g(l_2), g(l_3))$ is a real analytic diffeomorphism, it
follows that $u(\Omega)$ is bounded by a real analytic surface in
$g(J)^3$.

By the assumption $\tilde{F}(u)$ is convex in $u(\Omega)$. If $U$
is a connected component of $g(J)^3 - \overline{u(\Omega)}$, then
$\tilde{F}(u)$ is linear on $U$ since its partial derivatives are
constants on $U$ by the construction. Thus by corollary 2.6, the
result follows. \end{proof}

\section{ A Proof of Theorem 1.2}
The argument is essentially the same as that in \S3.3.  Recall
that $E$ is the set of all edges in the triangulated surface $(S,
T)$. If $ x \in \bold R^E$ and $i \in E$, we use $x_i$ to denote
$x(i)$. If $s \in T^{(2)}$ is a triangle with edges $i, j, k \in
E$, we denote it by $s=\{i,j,k\}$.

\subsection{A proof of theorem 1.2(3)}

Suppose otherwise that there exist two distinct hyperbolic
polyhedral metrics on $(S, T)$ so that their $\psi_h$ curvatures
are the same. Let $a=(a_i) \in \bold R^E$ be their common $\psi_h$
curvature.

Recall that a polyhedral metric on $(S,T)$ is given by its edge
length map $l: E \to \bold R_{>0}$. In using the variational
principle in theorem 4.1(4), the natural variable is given by $u:
E \to \bold R$ where $u(e) = g(l(e))$ with  $g(t) = \int^t_1
\coth^{h+1}(s/2) ds$. We call it the $u$-coordinate of the
polyhedral metric $l$ and we will use the $u$-coordinate to set up
the variational principle. Therefore there are two distinct points
(as $u$-coordinates)  $p \neq q \in g(\bold R_{>0})^E$ so that
their corresponding $\psi_h$ curvatures are the same $a \in \bold
R^E$. We will derive a contradiction by using the locally strictly
convex functions $F$ and its convex extension $\tilde{F}$
introduced in proposition 4.3 (associated to theorem 4.1(4)).

Define a $C^1$-smooth function $W: g(\bold R_{>0})^E \to \bold R$
by
$$ W(u) = \sum_{s \in T^{(2)}, s =\{i,j,k\}, i,j,k \in E} \tilde{F}(u_i,
u_j, u_k) - \sum_{ i \in E} a_i u_i.$$

The function $W$ is convex since it is a summation of convex
functions. Furthermore, by the definition of $W$, (4.1), the
definition of $\psi_h$ and $(a_i)$, $p$ and $q$ are both critical
points of $W$. Since $W$ is convex, $p$ and $q$ are both minimal
points of $W$. Furthermore, for all $t \in [0,1]$, $tp+(1-t)q$ are
minimal points of $W$. In particular, $$ W(tp + (1-t)q) = W(p)$$
for all $ t \in [0, 1]$. Since $$W(tp+(1-t)q) = \sum_{i,j,k \in E,
\{i,j,k\} \in T^{(2)}}f_{ijk}(t) - \sum_{ i \in E} a_i (tp_i +
(1-t) q_i)$$ where the function \begin{equation} f_{ijk}(t)
=\tilde{F}(tp_i + (1-t) q_i, tp_j + (1-t)q_j , tp_k +(1-t) q_k)
\end{equation} is convex, it follows that $f_{ijk}(t)$ is linear
in $t \in [0, 1]$.  Since $p \neq q$,  there exists a triangle
with edges $i,j,k \in E$ so that $(p_i,  p_j, p_k ) \neq (q_i,
q_j, q_k)$. Thus for $t \in [0, 1]$ close to $0$ or $1$, by
theorem 4.1 on the local strictly convexity, $f_{ijk}(t)$ is
strictly convex in $t$ near $0, 1$. This is a contradiction to the
linearity of $f_{ijk}(t)$.

\subsection{A proof of theorem 1.2(2)}
The proof is exactly the same as above using the extended convex
function $\tilde{F}$ in proposition 4.3 associated to theorem
4.1(2).

\subsection{A proof of theorem 1.2(1)}
The proof is the same as that in \S5.1 using the similarly defined
function $W$.  To be more precise, let $g(t) =- \frac{t^{-h}}{h}$
for $h \neq 0$ and $g(t) = \ln t$. By the same set up as in \S5.1,
we conclude that $f_{ijk}(t) $ given by (5.1) is linear in $t$. We
claim this implies that the two Euclidean polyhedral metrics
$u^{-1}(p)$ and $u^{-1}(q)$ differ by a scalar multiplication.
There are two cases to be discussed depending on $h=$ or $h \neq
0$.

CASE 1. $h=0$. In this case, $p-q \neq c(1,1,...,1)$ in $\bold
R^E$ for any constant $c$. By the connectivity of the surface $S$,
there exists a triangle with edges $i,j,k \in E$ so that $(p_i,
p_j, p_k) - (q_i, q_j, q_k) \neq (c,c,c)$ for any constant $c$. On
the other hand, by theorem 4.1(1), the action functional $F$ is
strictly locally convex in the hyperplane $u_1+u_2+u_3=0$ and $F(
u + (c,c,c) = F(u)$ for a scalar $c$ and $u \in u(\Omega)$. In
particular, this implies that the function $f_{ijk}(t)$ is
strictly convex in $t \in [0, 1]$  for $t$ close to $0$ or $1$.
This contradicts the linearity of  the function $f_{ijk}(t)$.

 CASE 2. $ h \neq 0$. In this case, $p \neq c q$ for any constant
 $c$. In particular, there exists a triangle with three edges
 $i,j,k \in E$ so that $(p_i, p_j, p_k) \neq c (q_i, q_j, q_k)$
 for any $c \in \bold R$.  By theorem 4.1(1), the function $f_{ijk}(t)$ is
strictly convex in $t \in [0, 1]$  for $t$ close to $0$ or $1$.
This contradicts the linearity of  the function $f_{ijk}(t)$.

Thus the two polyhedral metrics differ by a scaling.


\bibliographystyle{amsalpha}

\end{document}